\documentclass[reqno,tbtags,intlimits,a4paper,oneside,12pt]{amsart}
\usepackage[cp1251]{inputenc}%
\usepackage[english]{babel}
\usepackage{amssymb,upref,calc,url}
\usepackage[mathscr]{eucal}
\usepackage{graphicx}
\usepackage{amsmath,amscd,amsthm,verbatim}
\usepackage[all]{xy}
\usepackage[in]{fullpage}
\newtheorem{thm}{\hskip\parindent Theorem}[section]
\newtheorem{lem}[thm]{\hskip\parindent Lemma}
\newtheorem{prb}[thm]{\hskip\parindent Problem}
\newtheorem{cor}[thm]{\hskip\parindent Corollary}

\theoremstyle{definition}
\newtheorem{dfn}{\hskip\parindent Definition}[section]

\newtheorem{rem}[dfn]{\hskip\parindent Remark}
\newcommand{\R}{\mathcal{R}}
\DeclareMathOperator{\wt}{wt}

\begin{document}

\title{Rings of Coefficients of Universal Formal Groups for~Elliptic~Genus of~Level~N.}
\author{E.\,Yu.~Bunkova}
\address{Steklov Mathematical Institute of Russian Academy of Sciences, Moscow, Russia}
\email{bunkova@mi-ras.ru}
\thanks{Supported by the Russian Foundation for Basic Research (Grant No. 20-01-00157 A)}

\begin{abstract}
The paper is devoted to problems at the intersection of formal group theory, the theory of Hirzebruch genera, and the theory of elliptic functions.

The elliptic function of level~$N$ determines the elliptic genus of level~$N$ as a Hirzebruch genus.
It is known that the elliptic function of level~$N$
is a specialization of the Krichever function that determines the Krichever genus. The Krichever function is the exponential of the universal Buchstaber formal group. 

In this work we describe rings of coefficients of specializations of universal Buchstaber formal groups that correspond to the elliptic genus of~level~$N$ for $N = 2,3,4,5,6$.
\end{abstract}

\maketitle

\vspace{-10pt}

\section{Introduction} \label{s1}

All rings considered in this paper will be supposed to be commutative rings with~unity.
\textit{A one-dimensional formal group law} (or, briefly, a \textit{formal group}) over a ring $R$ is a formal series
\begin{equation} \label{FGL}
F(u,v)=u+v+\sum\limits_{i,j\geqslant 1}a_{i,j}u^iv^j\in R[[u,v]]
\end{equation}
that satisfies the conditions of associativity
\begin{equation} \label{ass}
F(u,F(v,w)) = F(F(u,v),w)
\end{equation}
and commutativity
$$F(u,v)=F(v,u).$$
For the general theory of formal groups see \cite{Hazewinkel1978}.

The \emph{exponential of a formal group} $F(u, v)$ over $R$ is a formal series
$f(x) \in R \otimes \mathbb{Q}[[x]]$, determined by the initial conditions $f(0)=0, f'(0) = 1$ and the addition law
\begin{equation}
\label{fxy}
f(x+y) = F(f(x),f(y)).
\end{equation}
Thus, the substitution $u = f(x)$, $v = f(y)$ linearizes any formal group $F(u,v)$ over $R \otimes \mathbb{Q}$.
The series $g(u)$ functionally inverse to $f(x)$ is called \emph{logarithm of the formal group} $F(u,v)$.

In the ring $R$ one can select a subring $R_F$ of coefficients for the formal group $F(u,v)$.
By definition, the ring $R_F$ is multiplicatively generated by the coefficients $a_{i,j}$ in \eqref{FGL}.
In~the case $R = R_F$ the formal group $F(u,v)$ over $R$ is called a \emph{generating formal group}. 

In this work we consider the problem of finding the rings of coefficients of a series of~formal groups that give specializations of the Buchstaber formal group (see Section~\ref{s5})
\begin{equation} \label{FB0}
F(u,v) =\frac{u^2 A(v) - v^2 A(u)}{u B(v) - v B(u)}, 
\end{equation}
where $A(u), B(u) \in R[[u]]$ and $A(0) = B(0) = 1$.
The exponentials of these specializations are elliptic functions of level $N$ that, by definition, are the functions that determine the elliptic genus of~level~$N$ as a Hirzebruch~genus.
The rings of coefficients for such formal~groups for $N=2,3$ were described in \cite{BU}, \cite{BBU}. In this work we give solutions to~this problem for $N = 4,5,6$. The main resuls are presented in Section \ref{e4}, where we give a detailed descriprion of such formal groups in the $N=4$ case.

The aim of this work is to find an approach for the $R$-integrability property that works for each elliptic genus of level $N$ (see Section \ref{s2}). In Sections \ref{s5} and \ref{s6} we present formal groups whose specializations give elliptic functions of level $N$. In Sections \ref{e2} and \ref{e3} we~present known results about universal formal groups of a given form with exponentials the~elliptic functions of level $2$ and of level $3$. Their rings of coefficients are of interest. We search for universal formal groups with coefficients over $\mathbb{Z}$ such that the formal groups are generating and the rings are torsion-free and have a simple structure.

In Section \ref{e4} we obtain the corresponding results for the elliptic function of level $4$.
In~Sections \ref{e5} and \ref{e6} we present some results for the elliptic functions of level $5$ and $6$.

\section{The Hirzebruch genus} \label{s2}

Let $K$ be an associative commutative unital algebra over $\mathbb{Q}$.
Let $l$ be a lattice in $\mathbb{C}$,
let $\sigma$, $\zeta$, $\wp$ be Weierstrass functions for this lattice (see \cite{Iso}, \cite{WW}).

The Hirzebruch genus is one of the most important classes of invariants of manifolds. A genus is a ring homomorphism
$L_f\colon \Omega_U \to K$ from the cobordism ring of stable complex manifolds $\Omega_U$ to $K$.
The Hirzebruch genus construction (see Sections 13 and 14 of \cite{CC}) matches such a genus to the series
\begin{equation} \label{fz}
f(z) = z + \sum_{k=1}^{\infty} f_k z^{k+1},  
\end{equation}
with coefficients~$f_k$ in~$K$.

There is an important property of $R$-integrability (see Section 14 of \cite{CC}),
that the image of $L_f$ belongs to the subring $R$ of the ring $K$.
There is a one-to-one correspondence between the set of formal groups over $R$ and the set of ring homomorphisms $\Omega_U \to R$.
Wherein the formal group $F$ corresponds to the ring homomorphism, determined by the Hirzebruch genus $L_f$, where $f$ is the exponential of this group with coefficients in $K = R \otimes \mathbb{Q}$.
Thus, to study the property of $R$-integrality of the Hirzebruch genus,
it is sufficient to 
construct the formal group corresponding to it and to study its ring of~coefficients.
If the ring of coefficients $R$ has torsion, then it's torsion-free subring gives the same exponential as the ring $K = R \otimes \mathbb{Q}$
is the same. Thus we will be interested in formal groups with torsion-free rings of coefficients.

In this work we consider formal groups that correspond to the Krichever genus.  By \cite{B19} (see also Section E.5 of \cite{T} and \cite{Krichever-90}), one can define the the Krichever genus as a~Hirzebruch genus where the Krichever function $f(z)$ satisfies a differential equation of~the form
\begin{equation} \label{feq}
f(z) f'''(z) - 3 f'(z) f''(z) = 6 q_1 f'(z)^2 + 12 q_2 f(z) f'(z) + 12 q_3 f(z)^2.
\end{equation}
In the non-degenerate case the Krichever function has the form (see \cite{Krichever-90}, \cite{Hohn})
\begin{equation} \label{fKr} 
{ \sigma(z) \sigma(\rho) \over \sigma(\rho - z)} \exp(\alpha z - \zeta(\rho) z).
\end{equation}
This expression depends on the variable $z$, the parameters $\alpha$, $\rho$ and the lattice $l$.

\vfill

\eject

A condition that the complex genus $L_f$ is fiberwise multiplicative with respect to~$\mathbb{C}P^{n-1}$
leads to the \emph{Hirzebruch functional equation} for the function $f(z)$:
\begin{equation} \label{Hfe} 
	\sum_{i = 1}^{n} \prod_{j \ne i} { 1 \over f(z_j - z_i)} = c,
\end{equation}
where $n$ is an integer greater than one, and $c$ is a constant
(see Chapter 4 of \cite{Hirtz} for the case of odd functions and Chapter 9 of \cite{T} for the general case).
The elliptic genus of~level~$N$ has this property for~$N|n$.
The fundamental results about the variety of~solutions of~the equation~\eqref{Hfe} were obtained in~\cite{Man}.
For $K = \mathbb{C}$ and~$n \leqslant 6$ the full classification of~solutions of~\eqref{Hfe} in the class of~series of the form \eqref{fz} was obtained in~\cite{Bun-18}.
The equation \eqref{Hfe} belongs to a wide class of functional equations considered in \cite{B18}
that have applications to the problem of Hirzebruch genera of complex cobordism classes.

The definition of an elliptic function of level $N$ in the general case is given in \cite{B19}. This is the function $f(z)$ determining the elliptic genus of level $N$ as a Hirzebruch genus. 
Note that the elliptic genus of level $2$ is the Ochanine--Witten genus (\cite{Osh}, \cite{Wit}). Next we recall the definitions in the non-degenerate case.

\begin{dfn}[{\cite[Appendix III, Section 1]{Hirtz}}] \label{d12}
A \emph{non-degenerate elliptic function of~level~$N$} with lattice~$l$
is a meromorphic function~$f$, such that
$f(0) = 0$, $f'(0) = 1$,
and~$G(z) = f(z)^N$ is an elliptic function with lattice~$l$ and divisor $N \cdot 0 - N \cdot \rho$ for~$\rho \in \mathbb{C}$.
Additionally, we require $N \in \mathbb{N}$ to be the minimal number such that $f(z)$ has this property.
\end{dfn}

From definition~\ref{d12} it follows that $\rho$ is an $N$-torsion point for the lattice~$l$, that is~$\rho \notin l$, $N \rho \in l$.
Moreover, for given $l$ and an $N$-torsion point $\rho$ one can obtain $G(z)$ and~$f(z)$ by definition~\ref{d12} in a unique way (see~\cite{Hirtz}).
From the minimality condition it follows that the order of $\rho$ as an element of the group $\mathbb{C}/l$ is equal to $N$.

\begin{lem}[{\cite[Lemma 2.3]{Bun-18}}] \label{llN}
For a non-degenerate elliptic function $f(z)$ of level $N$ with lattice~$l$ one can choose the generators $\omega, \omega'$ of the lattice $l$, such that
$\rho = \omega / N$ and the periodicity conditions hold 
\begin{align*}
f(z + \omega) &= f(z), & f(z + \omega') &= \exp\left(- 2 \pi I / N\right) f(z), \quad \text{where} \quad I^2 = -1.
\end{align*}
\end{lem}

\begin{cor}[\cite{Krichever-90}] \label{CorKri}
A non-degenerate elliptic function of level $N$ with parameters $\omega, \omega'$ is determined by the expression~\eqref{fKr} with lattice $l = \langle \omega, \omega' \rangle$ and $\alpha = \zeta(\omega/N) - 2 \zeta(\omega/2)/N$, $\rho = \omega/N$.
\end{cor}

\section{Universal formal groups} \label{s3}

A formal group $\mathcal{F}_U(u,v) = u + v + \sum \alpha_{i,j} u^i v^j$ over a ring $\R_U$ is called a
{\it universal formal group}
if for any formal group $F(u,v)$ over a ring $R$
there exists a unique ring homomorphism~$h_F\colon \R_U \to R$ such that $F(u,v) = u + v + \sum h_F(\alpha_{i,j}) u^i v^j$. 
The ring homomorphism~$h_F$ is called the \emph{classifying homomorphism}.
Further we suppose that the ring homomorphism $h_F$ is graded,
that is, the ring $R$ is graded and $\wt a_{i,j} = - 2 (i + j - 1)$.  
See Lemma 2.4 in \cite{BU} for the existence of the universal formal group.

When we consider formal groups of a given form, following \cite{BU}, we will speak of the {\it universal formal group of a given form}.
Namely, this is a formal group $\mathcal{F}_\#(u,v) = u + v + \sum \alpha_{i,j}^\# u^i v^j$ of this form over a ring $\R_\#$ such that for any formal group $F(u,v)$ of this form over a ring $R$
there exists a unique ring homomorphism~$h^\#_F \colon \R_\# \to R$ such that $F(u,v) = u + v + \sum h_F^\#(\alpha_{i,j}^\#) u^i v^j$. We have $h_F = h_F^\# h_\#$.

The ring of coefficients of the universal formal group of the form~\eqref{FB0} is described in~\cite{BU}. It uses a construction of the multiplicative generators in $\R_U$ that we will briefly describe here following \cite{BU}.

For natural numbers $m_1, \ldots, m_k$ we denote by $(m_1, \ldots, m_k)$ the greatest common divisor of these numbers. Using Euclidean algorithm, one can find integers $\lambda_1, \ldots, \lambda_k$ such that $\lambda_1 m_1 + \ldots \lambda_k m_k = (m_1, \ldots, m_k)$. For $n \geqslant 1$ set 
\[
 d(n+1) = \left(\begin{pmatrix}
                 n+1 \\ 1
                \end{pmatrix}, \begin{pmatrix}
                 n+1 \\ 2
                \end{pmatrix}, \ldots,
                \begin{pmatrix}
                 n+1 \\ n
                \end{pmatrix}.
 \right)
\]
By Kummer's theorem (see Theorem 9.2 in \cite{BU}) we have
\[
 d(n+1) = \left\{ \begin{matrix}
                   p, & \text{ if } n+1 = p^k \text{ where } p \text{ is a prime},  \\
                   1 & \text{ if not.}
                  \end{matrix}
 \right.
\]
For each $n$ we fix a set $(\lambda_1, \ldots, \lambda_n)$ such that the equality holds
\[
 \lambda_1 \begin{pmatrix}
                 n+1 \\ 1
                \end{pmatrix} + \lambda_2 \begin{pmatrix}
                 n+1 \\ 2
                \end{pmatrix} + \ldots + \lambda_n \begin{pmatrix}
                 n+1 \\ n
                \end{pmatrix} = d(n+1).
\]

For the universal formal group over the ring $\R_U$ we introduce the element $e_n \in \R_U$ using the~relation 
\begin{equation} \label{en}
 e_n = \lambda_1 \alpha_{1,n} + \lambda_2 \alpha_{2,n-1} + \ldots + \lambda_n \alpha_{n,1}. 
\end{equation}

For any formal group \eqref{FGL} over a ring $R$
we introduce the ideal $I^2$ in $R$ generated by the~elements of the form $a_{i_1, j_1} a_{i_2, j_2}$ in \eqref{FGL}, where $i_1, i_2, j_1, j_2 \geqslant 1$. In the case of a~generating formal group $F$ over $R$, the ideal $I^2$ is~the ideal of decomposible elements of~$R$.

\begin{dfn}[Definition 2.11 in \cite{BU}]
For a formal group $F$ over the ring $R$ the function $\rho: \mathbb{N} \to \{ \mathbb{N} \cup \infty \}$ is such that $\rho(n)$ is the order of the element $h_F(e_n)$ in the~ring~$R / I^2$.
\end{dfn}

\begin{rem}[Remark 4.4 in \cite{BU}]
For a generating formal group $F$ over a ring $R$ the~definition of the function $\rho$ does not depend on the coefficients in \eqref{en}. In this case the multiplicative generators of $R$ are the elements $h_F(e_n)$ such that $\rho(n)>1$.
\end{rem}

\section{Buchstaber formal group} \label{s5}

The \emph{Buchstaber formal group} \cite{Buc90} over the ring $R$ is a formal group of the form 
\begin{equation} \label{FB}
F(u,v) =\frac{u^2 A(v) - v^2 A(u)}{u B(v) - v B(u)}, 
\end{equation}
where $A(u), B(u) \in R[[u]]$ and $A(0) = B(0) = 1$.
Let us introduce the notations
\begin{align} \label{ABr}
A(u) &= 1 + \sum_{k=1}^\infty A_k u^k, & B(u) &= 1 + \sum_{k=1}^\infty B_k u^k. 
\end{align}
Note that the right hand side of \eqref{FB} does not depend on the coefficients $A_2$ and~$B_1$, therefore, further we set $A_2 = B_1 = 0$.
 
In~\cite{Buc90} (see also \cite[Theorem E~5.4]{T}) it was shown that the exponential of the universal Buchstaber formal group is the  Krichever function~\eqref{fKr}
(or it's degeneration).

\begin{lem}[Lemma 2.9 in \cite{BU}] \label{lr}
 Set $R_B = \mathbb{Z}[A_1, A_3, A_4, \ldots, B_2, B_3, B_4, \ldots]/I_{ass}$, where~$I_{ass}$ is the associativity ideal generated by the relation \eqref{ass}. Then the formal group~\eqref{FB} is~generating over $R_B$.
\end{lem}

\begin{cor}[Corollary 2.10 in \cite{BU}]
The ring homomorphism $h_B: \mathcal{R}_U \to R_B$, classifying the Buchstaber formal group over $R_B$, is an epimorphism. One can select the set of multiplicative generators $\{e_1, \ldots, e_n, \ldots \}$ in $\mathcal{R}_U$ such that the set of non-zero elements $h_B(e_n)$ is a minimal set of generators of $R_B$.
\end{cor}

\begin{thm}[Theorem 6.1 in \cite{BU}] \label{tBU}
The ring $R_B$ is multiplicatively generated by the elements $h_B(e_n)$, where $n = 1, 2, 3, 4$, $n = p^r$ (where $r \geqslant 1$ and $p$ is prime), and $n = 2^k-2$ (where $k>2$). For the function $\rho(n) = \rho_B(n)$ we have
\[
\rho_B(n) = \left\{ \begin{matrix}
                     \infty & \text{ for } n = 1,2,3,4,\\ 
                     p & \text{ for } n = p^r, n \ne 2,3,4, \text{ and } r \geqslant 1 ,\\
                     2 & \text{ for } n = 2^k-2, k \geqslant 3,\\ 
                     1 & \text{ in all other cases.}\\ 
                    \end{matrix}
 \right.
\]
The ring $R_B$ has torsion only of order two.
There exist multiplicative generators $e_{2^k-2}$, where $k\geqslant 3$, in $\R_U$ such that $h_B(e_{2^k-2})$ generates in $R_B$ the ideal $I_B$ of elements of order two. The ring $R_B/I_B$ is torsion-free.
\end{thm}

As we have mentioned in Section \ref{s2},
the elliptic genus of level $N$ is a special case of~the Krichever genus.
Thus the universal formal group that corresponds to the elliptic genus of~level $N$ can be defined as a specialization of the Buchstaber formal group.
The~equations that give this specializations for small $N$ are the following (see \cite{Ell4, Ust17}):
\begin{align*}
 &N=2\colon & 1 &= A(u); & &\\
 &N=3\colon & B(u) + 2 A'(0) u &= A(u)^2; & &\\
 &N=4\colon & \left(B(u) + {3 \over 2} A'(0) u\right)^2 &= A(u)^2 \left(A(u) - \left({3 \over 4} A'(0)^2 - B''(0)\right) u^2\right).  & &
\end{align*}
V.M. Buchstaber posed the problem to find such equations for arbitrary $N$.
In \cite{B19} this~problem has been reformulated in the following way:
\begin{prb}[\cite{B19}] \label{prob1}
Prove that:
\begin{enumerate}
 \item For the specialization of Buchstaber formal group, determined for
$n=\left[ {N-1 \over 2} \right]$, $m = \left[ {N -2 \over 2} \right]$ and arbitrary parameters
$(a_1, \ldots, a_m, b_1, \ldots, b_n)$, by the relation
\begin{multline} \label{GenEq}
(B(u) + b_1 u)^2 (B(u) + b_2 u)^2 \ldots (B(u) + b_{n-1} u)^2 (B(u) + b_n u)^{N-2n} = \\
= A(u)^2 (A(u) + a_1 u^2)^2 \ldots (A(u) + a_{m-1} u^2)^2 (A(u) + a_m u^2)^{N-1-2m},
\end{multline} 
the exponential can be only the elliptic function of level not greater than~$N$. 
\item The universal formal group corresponding to the elliptic genus of level~$N$ belongs to this specialization.
\end{enumerate}
\end{prb}
In \cite{B19} this problem was solved for $N=2,3,4,5,6$. 

\vfill
\eject

\section{Tate Formal Group} \label{s6}

The {\it general Weierstrass model for the elliptic curve} is given by the equation
\begin{equation} \label{A1}
Y^2 Z + \mu_1 X Y Z + \mu_3 Y Z^2 = X^3 + \mu_2 X^2 Z + \mu_4 X Z^2 + \mu_6 Z^3
\end{equation}
in homogeneous coordinates $(X:Y:Z)$
and parameters $\mu = (\mu_1, \mu_2, \mu_3, \mu_4, \mu_6)$.

A linear change of coordinates
\[
(X: Y: Z) \mapsto \left(X + {\nu_2 \over 2}\, Z: 2 Y + \mu_1 X + \mu_3 Z : Z\right), \quad \text{where} \quad \nu_2={\mu_1^2+4 \mu_2 \over 6},
\]
brings it into the {\it standard Weierstrass model for the elliptic curve}, which is
given in Weierstrass coordinates $(x:y:1)$ by the equation
\begin{equation} \label{eg}
 y^2 = 4 x^3 - g_2 x - g_3,
\end{equation}
where
\begin{equation} \label{g23}
 g_2 = 3 \nu_2^2- 2 \mu_1 \mu_3 - 4 \mu_4, \qquad g_3 = - \nu_2^3 + \nu_2 \mu_1 \mu_3 - \mu_3^2 + 2 \nu_2 \mu_4 - 4 \mu_6.
\end{equation}
The map $z \mapsto (x,y) = (\wp(z), \wp'(z))$, where $\wp(z)$ is the Weierstrass $\wp$-function with parameters~$g_2, g_3$,
gives an uniformization of the curve \eqref{eg} (see \cite{WW}). 

In Tate coordinates $(u:-1:s)$ (see~\cite{BBKrich, Tate-74})
the equation of the elliptic curve \eqref{A1} takes the form
\begin{equation} \label{Ats}
s =  u^3 + \mu_1 u s + \mu_2 u^2 s + \mu_3 s^2 + \mu_4 u s^2 + \mu_6 s^3.
\end{equation}
It can be seen directly from \eqref{Ats} that
the coordinate $s$ can be presented as a series $s(u) \in \mathbb{Z}[\mu_1, \mu_2, \mu_3, \mu_4, \mu_6][[u]]$. We have
\[
 s(u) = u^3 + \mu_1 u^4 + \left( \mu_2 + \mu_1^2 \right) u^5 
 + \left(\mu_3 + 2 \mu_1 \mu_2 + \mu_1^3\right) u^6 + \ldots
\]
Thus, the map $u \mapsto (u,s(u))$ determines a uniformization of the curve \eqref{Ats}. 

The classical geometric construction of the group structure on an elliptic curve  (see~\cite{Knapp}) in coordinates $(u, s(u))$ 
in the vicinity of $(0,0)$ determines a formal group $\mathcal{F}_{T}(u, v)$ over the ring $\mathbb{Z}[\mu_1, \mu_2, \mu_3, \mu_4, \mu_6]$,
which is called the {\it Tate formal group}. We denote by $\mathcal{R}_T$ the~subring $\mathcal{R}_T \subset \mathbb{Z}[\mu_1, \mu_2, \mu_3, \mu_4, \mu_6]$ of coefficients of $\mathcal{F}_T(u,v)$.

\begin{thm}[Theorem 2.3 in \cite{BBKrich}] \label{Adel} The addition law in the formal group $\mathcal{F}_{T}(u,v)$ is~given by the expression
\begin{equation} \label{FG}
\!\!\left(\!u + v - u v {\mu_1 + \mu_3 m + (\mu_4  + 2 \mu_6  m)k
\over 1 - \mu_3 k - \mu_6 k^2}\right)\!\!{ 1 + \mu_2 m + \mu_4 m^2 + \mu_6 m^3 \over (1 + \mu_2 n + \mu_4 n^2 + \mu_6 n^3) (1 - \mu_3 k - \mu_6 k^2)},
\end{equation}
\[
\text{where} \quad m = {s(u) - s(v) \over u - v}, \quad k = {u s(v) - v s(u) \over u - v}, \quad n = m + u v {1 + \mu_2 m + \mu_4 m^2 + \mu_6 m^3 \over 1 - \mu_3 k - \mu_6 k^2}.
\]
\end{thm}

\begin{cor}[Equation (3.3) in \cite{BBKrich}] \label{cor5.2}
\[
\left. {\partial \over \partial v} \mathcal{F}_{T}(u,v) \right|_{v = 0} = 1 - \mu_1 u - \mu_2 u^2 - 2 \mu_3 s(u) - 2 \mu_4 u s(u) - 3 \mu_6 s(u)^2.
\]
\end{cor}

\begin{thm}[Theorem 3.1 in \cite{BBKrich}] \label{thexp}
The exponential of the Tate formal group $\mathcal{F}_{T}(u, v)$ is determined by the series $f_T(z) \in \mathbb{Q}[\mu_1, \mu_2, \mu_3, \mu_4, \mu_6][[t]]$
that is the expansion at $t=0$ of the function
\begin{equation} \label{exp}
- 2 { \wp(z) -  {1 \over 12}(4 \mu_2 + \mu_1^2) \over \wp'(z)
- \mu_1 \wp(z) + {1 \over 12} \mu_1 (4 \mu_2 + \mu_1^2) - \mu_3},
\end{equation}
where the expressions for the parameters $g_2$ and $g_3$ of $\wp(z)$ are given by \eqref{g23}.
\end{thm}

\begin{prb}
 Find the ring of coefficients $\mathcal{R}_T$ of the universal Tate formal group.
\end{prb}

\section{Rings of coefficients for universal formal groups that correspond to~elliptic~genus~of~level~$2$} \label{e2}
\begin{thm}[{\cite{Buc90}}, see also Theorem 7.3 in {\cite{B19}}] \label{T2}
The elliptic function of level $2$ is~the~exponential of the universal formal group of the form
\begin{equation} \label{f2}
 F(u,v)=\frac{u^2 -v^2}{u B(v)-v B(u)}, \qquad B(0) = 1, \quad B'(0)=0. 
\end{equation}
\end{thm}

\begin{lem}
 Set $R_2 = \mathbb{Z}[B_2, B_3, B_4, \ldots]/I_{ass}$, where $I_{ass}$ is the associativity ideal generated by the relation \eqref{ass} for the formal group \eqref{f2}. Then the formal group \eqref{f2} is~generating over the ring $R_2$.
\end{lem}

\begin{proof}
The formal group \eqref{f2} is a specialization of the Buchstaber formal group \eqref{FB} determined by the relation
\begin{equation} \label{F2!}
A(u) = 1.
\end{equation}
Thus the Lemma follows from Lemma \ref{lr}.
\end{proof}

\begin{thm}[Theorem 8.2 in \cite{BU}] \label{t63}
The ring $R_2$ is multiplicatively generated by the elements $B_n$, where $n = 2^k$, $k \geqslant 1$. For the function $\rho(n) = \rho_2(n)$ we have 
\[
\rho_2(n) = \left\{ \begin{matrix}
                     \infty, & \text{ for } n = 2,4,\\
                     2, & \text{ for } n = 2^k, k\geqslant 3,\\
                     1 & \text{ in the other cases.} 
                    \end{matrix}
 \right.
\]
The ring $R_2$ is torsion-free.
\end{thm}

\begin{rem}
 As the ring $R_2$ has only one element of each weight, one can replace the generating elements $h_2(e_n)$ by $B_n$ in Theorem \ref{t63}.
\end{rem}

\begin{thm}[Theorem 7.4 in \cite{B19}] \label{T23}
The elliptic function of level $2$ is the exponential of the universal formal group,
which is a specialization of Buchstaber formal group \eqref{FB},
determined by the relations
\begin{align} \label{F2z}
A_1 &= 0, & A_3 &= 0.
\end{align} 
\end{thm}

\begin{cor}
 The ring of coefficients of the formal group from Theorem \ref{T23} is~$R_B / J_2$, where $J_2$ is the ideal generated by the relations \eqref{F2z}. The formal group is generating over~this ring.
\end{cor}

 Over the ring $R_B / J_2$ we have $\rho(n) = \rho_B(n)$ for $n \ne 1,3$ and $\rho(1) = \rho(3) = 1$.

\begin{cor}
For the classifying homomorphism 
$h_F\colon \R_U \to R_B / J_2$ the image of~the elements $e_n$ for  $n = p^r$ (where $r \geqslant 1$ and $p$ is prime, $p \ne 2$, $n \ne 3$) give elemets of~torsion~$p$ and the elements $e_n$ for $n = 2^k-2$ (where $k \geqslant 3$) 
give elements of torsion $2$.
\end{cor}

\begin{lem}[Lemma 9.7 in \cite{BBU}] \label{T2b}
The elliptic function of level $2$ is~the~exponential of the specialization of the Tate formal group \eqref{FG} determined by the relations
\begin{align*}
\mu_1 &= 0, & \mu_3 &= 0, & \mu_6 &=0. 
\end{align*}
Therefore it's ring of coefficients is a subring in $\mathbb{Z}[\mu_2, \mu_4]$. It is a formal group \eqref{f2} where the relation holds
\[
 B(u)^2 = (1 - \mu_2 u^2)^2 - 4 \mu_4 u^4.
\]
\end{lem}

\vfill
\eject

\section{Rings of coefficients for universal formal groups that correspond to~elliptic~genus~of~level~$3$} \label{e3}

\begin{thm}[Theorem 1 in {\cite{BB3}}] \label{T31}
The elliptic function of level $3$ is the exponential of the universal formal group of the form 
\begin{equation} \label{f3}
 F(u,v)=\frac{u^2 A(v) -v^2 A(u)}{u A(v)^2 - v A(u)^2}, \qquad A(0) = 1, \quad A''(0) = 0. 
\end{equation}
\end{thm}

\begin{lem}
 Set $R_3 = \mathbb{Z}[A_1, A_3, A_4, \ldots]/I_{ass}$, where $I_{ass}$ is the associativity ideal generated by the relation \eqref{ass} for the formal group \eqref{f3}. Then the formal group \eqref{f3} is~generating over the ring $R_3$.
\end{lem}

\begin{proof}
This formal group is a specialization of the Buchstaber formal group \eqref{FB}, determined by the relation 
\begin{equation} \label{F3!}
 B(u) + 2 A_1 u = A(u)^2.
\end{equation}
Thus the Lemma follows from Lemma \ref{lr}.
\end{proof}

\begin{thm}[Theorems 12.2 and 12.4 in \cite{BBU}]
The ring $R_3$ is multiplicatively generated by the elements $A_n$, where $n = 3^k$, $k \geqslant 1$.
For the function $\rho(n) = \rho_3(n)$, we have 
\[
 \rho_3(n) = \left\{ \begin{matrix}
                     \infty, & \text{ for } n = 1,3,\\
                     3, & \text{ for } n = 3^k, k\geqslant 2,\\
                     1 & \text{ in the other cases.} 
                    \end{matrix}
 \right.
\]
The ring $R_3$ is torsion-free.
\end{thm}

\begin{rem}
 As the ring $R_3$ has only one element of each weight, one can replace the generating elements $h_3(e_n)$ by $A_n$ in the Theorem. See Theorem 12.2 in \cite{BBU}.
\end{rem}

\begin{thm}[Theorem 7.8 in \cite{B19}] \label{T33}
The elliptic function of level $3$ is the exponential of the universal formal group,
which is a specialization of Buchstaber formal group \eqref{FB},
determined by the relations
\begin{align} \label{F3z}
B_2 &= A_1^2, & B_4 &= 0.
\end{align} 
\end{thm}

\begin{cor}
 The ring of coefficients of the formal group from Theorem \ref{T33} is~$R_B / J_3$, where $J_3$ is the ideal generated by the relations \eqref{F3z}. The formal group is generating over this ring.
\end{cor}

Over the ring $R_B / J_3$ we have $\rho(n) = \rho_B(n)$ for $n \ne 2,4$ and $\rho(2) = \rho(4) = 1$.

\begin{cor}
For the classifying homomorphism 
$h_F\colon \R_U \to R_B / J_3$ the image of~the elements $e_n$ for  $n = p^r$ (where $r \geqslant 1$ and $p$ is prime, $p \ne 3$, $n \ne 2,4$) give elemets of~torsion~$p$ and the elements $e_n$ for $n = 2^k-2$ (where $k \geqslant 3$) 
give elemets of torsion $2$.
\end{cor}

\begin{lem}[Lemma 9.8 in \cite{BBU}] \label{T3b}
The elliptic function of level $3$ is~the~exponential of the~specialization of the Tate formal group \eqref{FG} determined by the relations 
\begin{align*}
\mu_2 &= - \mu_1^2, & \mu_4 &= \mu_1 \mu_3,  & 3 \mu_6 &= - \mu_3^2.
\end{align*}
Therefore it's ring of coefficients is a subring in $\mathbb{Z}[\mu_1, \mu_3, \mu_6]/\{3 \mu_6 = - \mu_3^2\}$.

It is a formal group \eqref{f3} where the relation holds
\[
 A(u) = 1 - \mu_1 u - \mu_3 s(u).
\]
\end{lem}

\vfill
\eject

\section{Rings of coefficients for universal formal groups that correspond to~elliptic~genus~of~level~$4$} \label{e4}

\begin{thm}[Theorem 7.8 in {\cite{Ell4}}] \label{T41}
The elliptic function of level~$4$ is the exponential of the universal formal group of the form
\begin{equation} \label{f4}
 F(u,v)=\frac{u^2 A(v) -v^2 A(u)}{u B(v)-v B(u)},  
\end{equation}
where $A(u) = 1 + \sum_k A_k u^k$, $B(u) = 1 + \sum_k B_k u^k$, $B_1 = A_2 = 0$ and the relation holds
\begin{equation} \label{F4!}
\left(2 B(u) + 3 A_1 u\right)^2 = A(u)^2 \left(4 A(u) - \left(3 A_1^2 - 8 B_2\right) u^2\right). 
\end{equation}
\end{thm}

\begin{lem}
 Set $R_4 = \mathbb{Z}[A_1, A_3, A_4, \ldots, B_2, B_3, B_4, \ldots]/\{I_{ass}, I_4\}$, where $I_{ass}$ is the associativity ideal generated by the relation \eqref{ass} for the formal group \eqref{f4} and $I_{4}$ is the ideal generated by the relation \eqref{F4!}. Then the formal group \eqref{f4} is~generating over $R_4$.
\end{lem}

\begin{proof}
This formal group is a specialization of Buchstaber formal group \eqref{FB}, determined by the relation \eqref{F4!}. 
Thus the Lemma follows from Lemma \ref{lr}.
\end{proof}

\begin{thm}
The ring $R_4$ is multiplicatively generated by the elements $h_4(e_n)$,
where $n = 1, 2, 3, 4$, $n = 2^r$ (where $r \geqslant 3$), and $n = 2^r-2$ (where $r \geqslant 3$).
For the function $\rho(n) = \rho_4(n)$, we have 
\[
 \rho_4(n) = \left\{ \begin{matrix}
                     \infty, & \text{ for } n = 1,2,\\
                     4, & \text{ for } n = 3,\\
                     8, & \text{ for } n = 4,\\
                     2, & \text{ for } n = 2^r \text { and } n = 2^r - 2, \quad r \geqslant 3,\\
                     1 & \text{ in the other cases.} 
                    \end{matrix}
 \right.
\]
The ring $R_4$ has torsion of order two. The ideal of elements of order two is generated by the ideal $I_B$ (see Theorem \ref{tBU}) and the element $A_1^3 - 2 A_1 B_2 + B_3$. Factorizing by this ideal we obtain a torsion-free ring. 
\end{thm}

\begin{thm}[Theorem 7.12 in \cite{B19}] \label{T43}
The elliptic function of level $4$ is the exponential of the universal formal group,
which is a specialization of Buchstaber formal group \eqref{FB},
determined by the relations
\begin{align} \label{F4z}
2 A_3 &= - A_1 (A_1^2 - 2 B_2), & 8 B_4 &= 3 A_1^4 - 4 A_1^2 B_2 - 4 B_2^2.
\end{align} 
\end{thm}

\begin{cor}
 The ring of coefficients of the formal group from Theorem \ref{T43} is~$R_B / J_4$, where $J_4$ is the ideal generated by the relations \eqref{F4z}. The formal group is generating over this ring.
\end{cor}

Over the ring $R_B / J_4$ we have $\rho(n) = \rho_B(n)$ for $n \ne 3,4$ and $\rho(3) = 4$, $\rho(4) = 8$.

\begin{cor}
For the classifying homomorphism 
$h_F\colon \R_U \to R_B / J_4$ the image of~the elements $e_n$ for  $n = p^r$ (where $r \geqslant 1$ and $p$ is prime, $p \ne 2$, $n \ne 3$) give elemets of~torsion~$p$.
\end{cor}

The elliptic function of level $4$ can not be presented in the form $f_T(z)$ (see Theorem~\ref{thexp}). Thus we can not present it as an exponential of a specialization of the Tate formal group.

\vfill
\eject

\section{Rings of coefficients for universal formal groups that correspond to~elliptic~genus~of~level~$5$} \label{e5}

\begin{thm}[Theorem 7.15 in \cite{B19}] \label{T53}
The elliptic function of level $5$ is the exponential of the universal formal group,
which is a specialization of Buchstaber formal group \eqref{FB},
determined by the relations 
\begin{align} \label{F5z}
B_4 &= A_1^4 - 3 A_1^2 B_2 + B_2^2 + 4 A_1 A_3, \\
A_3^2 - 2 A_1 (5 A_1^2 - 2 B_2) A_3 &= (2 A_1^2 - B_2) (A_1^4 - 3 A_1^2 B_2 + B_2^2). \nonumber
\end{align} 
\end{thm}

\begin{cor}
 The ring of coefficients of the formal group from Theorem \ref{T53} is~$R_B / J_5$, where $J_5$ is the ideal generated by the relations \eqref{F5z}. The formal group is generating over this ring.
\end{cor}

\begin{thm}[Theorem 7.14 in \cite{B19}] \label{T5}
The elliptic function of strict level $5$ is the exponential of the universal formal group,
which is a specialization of Buchstaber formal group \eqref{FB},
determined by the condition that there exist $(a_1, b_1, b_2)$, such that
\begin{equation} \label{F5!}
\left(B(u) + b_1 u\right)^2 (B(u) + b_2 u)  = A(u)^2 \left(A(u) + a_1 u^2\right)^2, 
\end{equation} 
but the relations \eqref{F2!}, \eqref{F3!} or \eqref{F4!} do not hold.
\end{thm}

Consider the formal group over the ring $R$ of the form
\begin{equation} \label{f5}
 F(u,v)=\frac{u^2 A(v) -v^2 A(u)}{u B(v)-v B(u)},  
\end{equation}
where $A(u) = 1 + \sum_k A_k u^k$, $B(u) = 1 + \sum_k B_k u^k$, $B_1 = A_2 = 0$ and the relation \eqref{F5!} holds, where $A_k \in R$, $B_k \in R$, $a_1 \in R, b_1 \in R, b_2 \in R$.

Set $R_5 = \mathbb{Z}[a_1, b_1, b_2, A_1, A_3, A_4, \ldots, B_2, B_3, B_4, \ldots]/\{I_{ass}, I_5\}$, where $I_{ass}$ is the associativity ideal generated by the relation \eqref{ass} for the formal group \eqref{f5} and $I_{5}$ is the ideal generated by the relation \eqref{F5!}. The formal group \eqref{f5} over the ring $R_5$ is by condition the universal formal group of this form. 
 The formal group \eqref{f5}
is~not generating over $R_5$.

\vfill
\eject

\section{Rings of coefficients for universal formal groups that correspond to~elliptic~genus~of~level~$6$} \label{e6}

\begin{thm}[Theorem 7.18 in {\cite{B19}}] \label{T63}
The elliptic function of level $6$ is the exponential of the universal formal group,
which is a specialization of Buchstaber formal group \eqref{FB},
determined by the relations
\begin{align} 
& 5 A_1^5 - 20 A_1^3 B_2 + 15 A_1 B_2^2 + 24 A_1^2 A_3 - 24 B_2 A_3 - 18 A_1 B_4 = 0, \nonumber \\
& 5 A_1^6 - 21 A_1^4 B_2 + 18 A_1^2 B_2^2 + 48 A_1^3 A_3 - 36 A_1 B_2 A_3 + 18 A_3^2 = 0, \label{F6z} \\
& 3 A_3 (11 A_1^4 - 56 A_1^2 B_2 - 18 A_1 A_3 + 33 B_2^2 + 18 B_4) = \nonumber \\
& \qquad \qquad \qquad \qquad = - A_1 (5 A_1^2 - 6 B_2) (A_1^2 - 3 B_2) (A_1^2 - 4 B_2), \nonumber\\
& 6 A_1^2 (A_1^2 - B_2) (7 A_1^2 - 4 B_2)^2 = (17 A_1^4 - 18 A_1^2 B_2 + 6 A_1 A_3 + 3 B_2^2 + 6 B_4)^2. \nonumber
\end{align} 
\end{thm}

\begin{cor}
 The ring of coefficients of the formal group from Theorem \ref{T63} is~$R_B / J_6$, where $J_6$ is the ideal generated by the relations \eqref{F6z}. The formal group is generating over this ring.
\end{cor}

\begin{thm}[Theorem 7.17 in {\cite{B19}}] \label{T6}
The elliptic function of strict level $6$ is the exponential of the universal formal group,
which is a specialization of Buchstaber formal group \eqref{FB},
determined by the condition that there exist parameters $(a_1, a_2, b_1, b_2)$, such that
\begin{equation} \label{F6!}
(B(u) + b_1 u)^2 (B(u) + b_2 u)^2  = A(u)^2 (A(u) + a_1 u^2)^2 (A(u) + a_2 u^2),
\end{equation} 
but the relations \eqref{F2!}, \eqref{F3!}, \eqref{F4!}, or \eqref{F5!}
do not hold.
\end{thm}

Consider the formal group over the ring $R$ of the form
\begin{equation} \label{f6}
 F(u,v)=\frac{u^2 A(v) -v^2 A(u)}{u B(v)-v B(u)},  
\end{equation}
where $A(u) = 1 + \sum_k A_k u^k$, $B(u) = 1 + \sum_k B_k u^k$, $B_1 = A_2 = 0$ and the relation \eqref{F6!} holds, where $A_k \in R$, $B_k \in R$, $a_1 \in R, a_2 \in R, b_1 \in R, b_2 \in R$.

Set $R_6 = \mathbb{Z}[a_1, a_2, b_1, b_2, A_1, A_3, A_4, \ldots, B_2, B_3, B_4, \ldots]/\{I_{ass}, I_6\}$, where $I_{ass}$ is the associativity ideal generated by the relation \eqref{ass} for the formal group \eqref{f6} and $I_{6}$ is the ideal generated by the relation \eqref{F6!}. The formal group \eqref{f6} over the ring $R_6$ is by condition the universal formal group of this form. 
 The formal group \eqref{f6}
is~not generating over $R_6$.

\vfill

\eject

\end{document}